\documentclass[11pt]{article}
\usepackage{amsmath,amsthm,amsfonts,amssymb,bbm}
\usepackage[sort, numbers, square]{natbib}
\usepackage{hyperref}

\usepackage{enumerate}
\usepackage[all]{xy}
\usepackage{color} 
\definecolor{taucol}{gray}{0.5}
\newcommand{\RR}{\ensuremath{\mathbb{R}}}
\newcommand{\NN}{\ensuremath{\mathbb{N}}}
\newcommand{\PP}{\ensuremath{\mathbb{P}}}
\newcommand{\EE}{\ensuremath{\mathbb{E}}}
\newcommand{\D}{\ensuremath{\text{d}}}
\newcommand{\Half}{\ensuremath{{\mathcal H}}}
\newcommand{\DepSet}{\ensuremath{{\mathcal K}}}
\newcommand{\PotentialDepSet}{\ensuremath{{\mathcal L}}}
\newcommand{\Bin}{\ensuremath{\text{Bin}}}
\newcommand{\pmid}{\ensuremath{\,:\,}}
\newcommand{\finite}{\ensuremath{\mathcal{F}}}
\newcommand{\eins}{\ensuremath{\mathbf{1}}}
\newcommand{\Eins}{\ensuremath{\mathbbm{1}}}
\DeclareMathOperator{\id}{id}
\DeclareMathOperator{\Cov}{Cov}
\newcommand{\distr}{\ensuremath{\mathcal{L}}}
\newcommand{\punkt}{\phantom{x}}
\newcommand{\cl}{\ensuremath{\text{\normalfont closure}}}
\newcommand{\maxecf}{\ensuremath{\Theta_{\text{\normalfont max}}}}
\newcommand{\binecf}{\ensuremath{\Theta_{\text{\normalfont bin}}}}
\newcommand{\limecf}{\ensuremath{\Theta_{\text{\normalfont lim}}}}
\newcommand{\maxcorr}{\ensuremath{\Xi_{\text{\normalfont max}}}}
\newcommand{\bincorr}{\ensuremath{\Xi_{\text{\normalfont bin}}}}
\newcommand{\limcorr}{\ensuremath{\Xi_{\text{\normalfont lim}}}}
\newtheorem{theorem}{Theorem}
\newtheorem{proposition}[theorem]{Proposition}
\newtheorem{corollary}[theorem]{Corollary}

\theoremstyle{definition}
\newtheorem{definition}[theorem]{Definition}

\theoremstyle{remark}
\newtheorem{remark}[theorem]{Remark}

\begin{document}

\title{Characterizing extremal coefficient functions\\ and extremal correlation functions} 

\author{
  Kirstin Strokorb\footnote{Institute for Mathematical Stochastics, Georg-August-University  G\"ottingen, Goldschmidtstr. 7, 37077 G\"ottingen, Email: strokorb@math.uni-goettingen.de},
  Martin Schlather\footnote{Institute of Mathematics, University of Mannheim, 68131 Mannheim, Email: schlather@math.uni-mannheim.de}
}

\maketitle 

\begin{abstract}
We focus on two dependency quantities of a max-stable random field $X$ on some space $T$: the extremal coefficient function $\theta$ which we define on finite sets of $T$ and the extremal correlation function $\chi(s,t)=\lim_{x \uparrow \infty} \PP(X_s \geq x \mid X_t \geq x)$.  We fully characterize extremal coefficient functions $\theta$ by a property called complete alternation and construct a corresponding max-stable random field. Simple properties and consequences concerning the convex geometry of extremal coefficients are derived. We study how the continuity of $X$, $\theta$ and $\chi$  are linked to each other, and we show that extremal correlation functions $\chi$ allow for convex combinations in general, and for products and pointwise limits if the resulting function is continuous. These are operations which are well-known for positive definite functions, but the latter are non-trivial for extremal correlation functions. Finally, we regard some additional implications, when the random field $X$ on $T=\RR^d$ is stationary.
\end{abstract}

\phantom{\textit{2010 MSC}:} {Secondary 60G60,60E15}



\section{Introduction} 

Max-stable random fields have gained increasing attention in  many areas of application such as hydrology or meteorology --  see e.g. \cite{blanchetdavison_11,buishandetalii_08,naveauetalii_09} for some recent examples. The structure theory and classification of max-stable processes is well-studied, exemplarily we refer to \cite{dehaan_84,ginehahnvatan_90,wangstoev_10}. However, appropriate modelling of the spatial dependence structure is still a difficult question.
In contrast to Gaussian random fields, the dependence structure of a finite sample of a max-stable random field cannot be described parametrically. Instead, 
the notion of {\em spectral measure}, {\em stable tail dependence function} or {\em dependency set} (cf. \cite{beirlantetalii_03,dehaan_84,molchanov_08,resnick_08}, for example) provide three equivalent concepts each of which fully describes the dependence structure of a multivariate max-stable distribution with fixed marginals. Further, {\em Pickands dependence function} \cite{pickands_81} comprises the complete dependence structure of a bivariate max-stable distribution.

However, all of these dependency descriptors may have a  very complex structure which makes them hard to be estimated from data (cf. \cite{dreeshuang_98,einmahlsegers_09,einmahlkrajinasegers_08} and references therein for some approaches on the issue in a multivariate setting). 

Therefore, several attempts have been made to propose meaningful, but simpler dependence measures:
In this paper, we focus on {\em extremal coefficients} \cite{smith_90} in a spatial setting and a function which we call {\em extremal correlation function} and was suggested in \cite{schlathertawn_03} as an analogue to the correlation function for extreme values. The latter is also a special case of the so-called {\em extremogram} \cite{davismikosch_09} and sometimes referred to as {\em (upper) tail dependence coefficient (function)} \cite{beirlantetalii_03,davismikosch_09,falk_05}, {\em $\chi$-measure} \cite{beirlantetalii_03,colesheffernantawn_99} or {\em extremal coefficient function} \cite{fasenetalii_10}. 
Besides being a bivariate dependence measure, the extremal correlation function also contains mixing information of the process \cite{kabluchkoschlather_10}.
An estimator which stems from a geostatistical approach is proposed in \cite{naveauetalii_09}.

The purpose of this paper is to characterize these spatially defined quantities --  extremal coefficients and extremal correlation function -- through their properties, to give necessary and sufficient criteria and to establish operations on the class of such quantities.

The text is organized as follows: In section \ref{foundations} we give the definitions and our  notation of the basic objects that we study, i.e. max-stable random fields and related quantities such as extremal coefficient function and extremal correlation function are introduced.  In section \ref{harmonic} we characterize extremal coefficient functions by a concept called complete alternation and construct a corresponding max-stable random field, thus generalizing results of \cite{schlathertawn_02} to a spatial setting. After deriving some immediate consequences from this,  section \ref{convexgeometry} is an intermezzo towards the convex geometry related to the results of section \ref{harmonic}.
In section \ref{continuity} we analyse how the continuity of extremal coefficient function, extremal correlation function and corresponding max-stable random fields are mutually linked. Section \ref{inclusions} classifies the convex sets of extremal coefficient functions and extremal correlation functions within other related convex sets of functions. We shall see that some of these convex sets of positive definite functions are additionally closed under pointwise limits and products. Finally, we present some additional conclusions which apply to stationary max-stable random fields on $\RR^d$ in section \ref{stationary}.

\section{Foundations and Definitions}\label{foundations}

We consider max-stable random fiels $X=\{X_t\}_{t \in T}$ on an arbitrary location space $T$ with the same one-dimensional non-degenerate marginals. For convenience we use standard Fr{\'e}chet marginals, i.e. $\PP(X_t \leq x)=e^{-1/x}$ for $t \in T$ and $x \geq 0$. Then, max-stability of the random field $X$ simply means that for any $n$ independent copies $X^{(1)},\dots,X^{(n)}$ of $X$ we have $\distr(n^{-1} \max_{i=1,\dots,n}(X^{(i)})) \sim \distr(X)$, i.e. these fields are equally distributed.

We will frequently fix a finite subset $M \subset T$ and use the following conventions: Elements of $\RR^M$, the vector space of real-valued  functions on $M$, are denoted by $x=(x_t)_{t \in M}$ where $x_t=x(t)$. By $\eins^M_L \in \RR^M$, we denote the indicator function of $L$ where $\emptyset \neq L \subset M$. Thus, the elements $\{\eins^M_{\{t\}}\}_{t \in M} \in \RR^M$ form a standard basis of $\RR^M$.
The standard scalar product on this space is $\langle x,y \rangle = \sum_{t \in M} x_t \cdot y_t$.  Analogous notation is made for the space $[0,\infty]^M$, the space of $[0,\infty]$-valued functions on $M$. In this context, we define the expressions $1/\infty:=0$ and $1/0:=\infty$. Further, we consider some reference norm $\lVert \cdot \rVert$ on $\RR^M$ (not necessarily the one from the scalar product) and denote $S_M:=\{ a \in [0,\infty]^M \pmid \lVert a \rVert = 1 \}$.

The full dependence structure of each multivariate distribution of $\{X_t\}_{t \in M}$ for some finite subset $M \subset T$  can be described by using one of the following equivalent approaches:

\begin{itemize}
\item {\em spectral measure} $H_M$ (cf. \cite{dehaan_84,resnick_08}, for example), i.e. the measure on  $S_M$ such that for $x \in [0,\infty]^M$
\begin{align}\label{spectralmeasure}
- \log \PP(X_{t} \leq x_t, \, t \in M) = \int_{S_M} \max_{t \in M}\left(\frac{a_t}{x_t}\right) H_M (\D a)
\end{align}
\item {\em stable tail dependence function} $\ell_M$ (cf. \cite{beirlantetalii_03}, for example), i.e. the function on  $[0,\infty]^M$ defined through
\begin{align*}
\ell_M(x) := - \log \PP(X_{t} \leq 1/x_t, \, t \in M) = \int_{S_M} \max_{t \in M} \left(a_t \cdot x_t\right) H_M (\D a)
\end{align*}
\item {\em dependency set} $\DepSet_M$ (cf. \cite{molchanov_08}), i.e. the unique compact convex set $\DepSet_M \subset [0,\infty)^M$ satisfying 
\begin{align}\label{depSet}
  \ell_M(x)=\sup\{ \langle x,y \rangle \pmid y \in \DepSet_M\},
\end{align}
where it suffices to check this condition for  $0\neq x \in [0,\infty)^M$. 
\end{itemize}

Of course, if we consider such quantities for different subsets $M \subset T$, they have to be consistent among each other to define a random field $X=\{X_t\}_{t \in T}$ on $T$. For instance, due to our choice of standard Fr{\'e}chet marginals we have the normalization $\ell_M(\eins^M_{\{t\}})=\int_{S_M}a_t H_M (\D a)=1$ for $t \in M$.

For any non-empty finite set of locations $M \subset T$ its {\em extremal coefficient} is defined using some $x \in (0,\infty)$ by 
\begin{align}\label{theta_intro}
  \theta(M):= \frac{- \log \PP(\max_{t \in M} X_t \leq x)}{- \log \PP(X_t \leq x)} = \int_{S_M} \max_{t \in M}\left({a_t}\right) H_M (\D a) = \ell_M(\eins^M_M).
\end{align}
Indeed, $\theta(M)$ does not depend on $x \in (0,\infty)$ and takes values in the interval $[1,|M|]$. Roughly speaking, the extremal coefficient $\theta(M)$ grasps the effective number of independent variables among the random variables $\{X_t\}_{t \in M}$. It is coherent to set $\theta(\emptyset):=0$.
Consistency of the finite dimensional dependency measures requires also that 
$\theta(L)=\ell_L(\eins^L_L)=\ell_M(\eins^M_L)$ for $\emptyset \neq L \subset M$, for example.
We will consider $\theta$ in this paper as a function on the {\em set of finite subsets of $T$}, which we denote by $\finite(T)$. The function 
\begin{align*}\theta: \finite(T) \rightarrow [0,\infty)\end{align*} 
is then called {\em extremal coefficient function}. It is an application of 
l'H{\^o}pitals rule that $\theta(M)$ coincides with
\begin{align}\label{theta_lHopital}
\theta(M)&= \lim_{x \uparrow  \infty}
    \frac{1-\PP\left(\max_{t \in M} X_t \leq x \right)}{1-\PP\left(X_t \leq x \right)}
    = \lim_{x \uparrow  \infty}
    \frac{\PP\left(\exists \, t \in M  \text{ such that } X_t \geq x \right)}{\PP\left(X_t \geq x \right)}.
\end{align}
 That is why the extremal correlation function $\theta$ is intimately connected with a bivariate characteristic of the max-stable random field  $X$ which we call {\em extremal correlation function} $\chi$ given by the following limit
\begin{align} \label{chi_intro}
\chi: T \times T \rightarrow [0,1] \qquad
  \chi(s,t):= \lim_{x \uparrow \infty} \PP(X_s \geq x \mid X_t \geq x) 
\end{align}
This quantity dates back to \cite{sibuya_60} and has entered the literature under different names, most prominently (upper) tail dependence coefficient (function) \cite{beirlantetalii_03,davismikosch_09,falk_05} or $\chi$-measure \cite{beirlantetalii_03,colesheffernantawn_99}. We prefer the term extremal correlation function, since $\chi$ is a symmetric positive definite kernel on $T \times T$ as has been observed already in \cite{davismikosch_09,schlathertawn_03}, for example. The relation
\begin{align}\label{sumistwo}
\theta(\{s,t\}) + \chi(s,t) = 2 \qquad \forall\, s,t \in T
\end{align}
holds because of (\ref{theta_lHopital}) and is widely known \cite{beirlantetalii_03,colesheffernantawn_99,fasenetalii_10}.
For technical reasons we introduce
\begin{align} \label{eta_intro}
\eta: T \times T \rightarrow [0,1] \qquad
\eta(s,t) := 1-\chi(s,t) = \theta(\{s,t\})-1.
\end{align}
In terms of the spectral measure, we can express these bivariate functions as
\begin{align*}
\chi(s,t) = \int_{S_{M}} \min \left({a_s,a_t}\right) H_{M} (\D a)  \qquad  \eta(s,t) =  \frac{1}{2} \int_{S_{M}}  \lvert {a_s-a_t} \rvert \, H_{M} (\D a).
\end{align*}
for any finite $M\subset T$ with $s,t \in M$.

\section{Harmonic analysis of extremal coefficients}\label{harmonic}

\subsection{A consistent max-linear model for max-stable random fields}

Suppose we are given a collection of real numbers $(\tau^M_L)$ where $M \subset T$ ranges over all non-empty finite subsets of $T$ and $L$ over all non-empty subsets of $M$. We consider random fields $X^*=\{X^*_t\}_{t \in T}$ on $T$ with marginal distributions of the following form: 
\begin{align}\label{model}
- \log \PP(X^*_t \leq x_t, \, t \in M) = \sum_{\emptyset \neq L \subset M} \tau_{L}^{M} \max_{t \in L}\frac{1}{x_t}
&& \qquad \forall \, M \in \finite(T)\setminus\{\emptyset\}, \, x \in [0,\infty]^M
\end{align}
Remember that $\finite(T)$ denotes the set of finite subsets of $T$.
Such a model has been studied already in \cite{schlathertawn_02} for finite sets $T$.
We draw our attention now to arbitrary index sets $T$ and list the conditions on these numbers $(\tau^M_L)$ which guarantee that they define the distribution of a max-stable random field $X^*$ with standard Fr{\'e}chet marginals:

First of all observe that once these numbers $(\tau^M_L)$ define a valid random field via model (\ref{model}), max-stability  is guaranteed immediately. As to the other conditions: 
\begin{itemize}
\item   Formula (\ref{model}) defines a valid  distribution function  for all non-empty finite $M \subset T$ if
\begin{align}
\label{tau1}
\tau^M_L &\geq 0 
&& \qquad \forall \, M \in \finite(T)\setminus\{\emptyset\}, \, \emptyset \neq L \subset M
\end{align} 
In case these inequalities are satisfied, the distribution of $\{X^*_t\}_{t \in M}$ coincides with the distribution of the random vector $\{X_t\}_{t \in M}$ which is given as a max-linear model built from $2^{|M|-1}$ i.i.d. standard Fr{\'e}chet variables $\{ V_L \}_{\emptyset \neq L \subset M}$ by $X_t = \max_{t \in L \subset M} \tau_{L}^{M} V_L $ (cf. \cite{schlathertawn_02}).
\item Subsequently, Kolmogorov consistency is ensured if and only if
\begin{align}
\label{tau2}
\tau^M_L=\tau^{M \cup \{t\}}_L+\tau^{M \cup \{t\}}_{L \cup \{t\}}
&& \qquad \forall \, M \in \finite(T)\setminus\{\emptyset\}, \, \emptyset \neq L \subset M,\, t \in T, \, t \not\in M,
\end{align}
which is equivalent to 
\begin{align}
\label{tau4}
\tau^A_K = \sum_{ J \subset M \setminus A} \tau^M_{K \cup J} 
&& \qquad \forall \, M \in \finite(T)\setminus\{\emptyset\}, \emptyset \neq K \subset A \subset M.
\end{align}
\item The model has standard Fr{\'echet} marginals if and only if
\begin{align}
\label{tau3}
\tau^{\{t\}}_{\{t\}}&=1
&& \qquad \forall \, t \in T. 
\end{align}
\end{itemize}

To illustrate the model (\ref{model}) better, we take a look at its spectral measure: Since it is  a max-linear model, the spectral measure $H^*_M$ of $\{X^*_{t}\}_{t \in M}$ is discrete and can be written as
\begin{align*}
H^*_M  = \sum_{\emptyset \neq L \subset M} \tau^M_L \lVert \eins^M_L \rVert \, \delta_{\frac{\eins^M_L}{\lVert \eins^M_L \rVert}}
\end{align*}
with respect to the reference norm $\lVert \cdot \rVert$ on $\RR^M$ (cf. (\ref{spectralmeasure})), see also figure \ref{spectralmeasurefigure}.

\begin{figure}[htp]
\centerline{
\xymatrix@=25pt@*=0{&& \ar@{<-}^(.2){\eins^{123}_3}[dd] \\ \\
&& { \bullet}  \ar@{-}^(-.2){\tau^{123}_{3}}[rr]\ar@{--}[dd] 
& & { \bullet}  \ar@{-}[dd] \ar@{}_(-.25){\tau^{123}_{23}}[ll]
\\
&{ \bullet} \ar@{-}[ur]\ar@{-}^(-.2){\tau^{123}_{13}}[rr]\ar@{-}[dd]
& & { \bullet} \ar@{-}[ur]\ar@{-}[dd] \ar@{}_(.1){\tau^{123}_{123}}[ur]
\\
&& {\circ} \ar@{--}[rr]
& & { \bullet} \ar@{}_(-.25){\tau^{123}_{2}}[ll] && \ar@{<-}^(.2){\eins^{123}_2}[ll] 
\\
&{ \bullet} \ar@{-}^(-.2){\tau^{123}_{1}}[rr]\ar@{--}[ur]
& & { \bullet} \ar@{-}[ur]  \ar@{}_(.1){\tau^{123}_{123}}[ur]\\ 
\ar@{<-}_(.2){\eins^{123}_1}[ur]
}
}
\caption{
Spectral measure representation of $\{X^*_{t}\}_{t \in M}$ for $M = \{1,2,3\}$ if we choose the reference norm on $\RR^M$ to be the maximum norm. In this case the spectral measure  simplifies to a sum of weighted pointmasses on the vertices of a cube: $H^*_M  = \sum_{\emptyset \neq L \subset M} \tau^M_L \, \delta_{\eins^M_L}$. 
}\label{spectralmeasurefigure}
\end{figure}
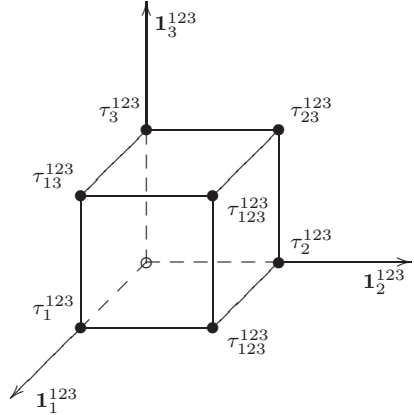

\subsection{Complete alternation of extremal coefficients}

In \cite{molchanov_08} and \cite{schlathertawn_02}  the class of possible extremal coefficients has been fully characterized in a multivariate setting. This has been done by a set of inequalities which  means the complete alternation of those coefficients and immediately transfers to the spatial setting as we shall see next. Behind the following characterizations, we use the fact that $\finite(T)$, the set of finite subsets of $T$, forms a semigroup with respect to the union operation $\cup$ and with neutral element the empty set $\emptyset$.
The following notation is mainly adopted from \cite{molchanov_05} and  \cite{bcr_84}.   For a function $f:\finite(T) \rightarrow \RR$ and elements $K,L \in \finite(T)$ we set:
\begin{align*}
\left(\Delta_{K}f\right)(L):= f(L)-f(L\cup K)
\end{align*}

\begin{definition} (complete alternation / negative definiteness)\\
A function $\psi: \finite(T) \rightarrow \RR$ is called {\em completely alternating} on $\finite(T)$
if for all $n \geq 1$, $\{K_1,\dots, K_n\} \subset \finite(T)$ and $K \in \finite(T)$
\[ 
\left(\Delta_{K_1}\Delta_{K_2} \dots \Delta_{K_n} \psi \right)(K) = 
\sum_{I \subset \{1, \dots, n\}} (-1)^{|I|} \,\psi\left(K \cup \bigcup_{i \in I} K_i\right)
 \leq 0.
\]
A function $\psi: \finite(T) \rightarrow \RR$ is called {\em negative definite (in the semigroup sense)} on $\finite(T)$
if for all $n \geq 2$, $\{K_1,\dots, K_n\} \subset \finite(T)$ and $\{a_1,\dots,a_n\} \subset \RR$ with $\sum_{j=1}^n a_j=1$
\[ 
\sum_{j=1}^n \sum_{k=1}^n a_j a_k \psi(K_j \cup K_k) \leq 0.
\]
Because the semigroup $(\finite(T),\cup,\emptyset)$ is idempotent, these two terms coincide, i.e. $\psi: \finite(T) \rightarrow \RR$ is {\em completely alternating} if and only if $\psi$ is {\em negative definite (in the semigroup sense)} (cf. \cite{bcr_84}, 4.1.16).
\end{definition}

An important example of a completely alternating function on $\finite(T)$ is the {\em capacity functional} $C:\finite(T) \rightarrow \RR$ of a binary random field $Y$. It is given by 
\begin{align}\label{cap}
C(M) = \PP(\exists \, t \in M \text{ such that } Y_t = 1) = \PP\left(\max_{t \in M} Y_t = 1\right).
\end{align}
The following theorem characterizes such capacity functionals:

\begin{theorem} [\cite{molchanov_05}, p. 52, for example]   \label{cap_CA} 
  The function $C: \finite(T) \rightarrow \RR$ is the capacity of a binary $\{0,1\}$-field on $T$ with the same one-dimensional marginals if and only if the following four conditions are satisfied:
  \begin{enumerate}[(i)]
  \item $C$ takes values in the interval $[0,1]$.
  \item $C(\emptyset) = 0$
  \item $C(\{t\}) =: p \leq 1$ is independent of $t \in T$.
  \item $C$ is completely alternating.
  \end{enumerate}
\end{theorem}

Surprisingly, we may characterize the possible extremal coefficient functions on $\finite(T)$ quite similarly:

\begin{theorem} \label{ecf_CA}
  \begin{enumerate}[a)]
  \item
    The function $\theta: \finite(T) \rightarrow \RR$ is the extremal coefficient function of a max-stable random field on $T$ with standard Fr{\'e}chet marginals if and only if the following three conditions are satisfied:
    \begin{enumerate}[(i)]
    \item  $\theta(\emptyset) = 0$
    \item  $\theta(\{t\}) = 1$ for all $t \in T$
    \item  $\theta$ is completely alternating.
    \end{enumerate}
  \item \label{ecf_process} If these conditions are satisfied, the following choice of coefficients 
    \begin{align*}
      \tau^M_L :=  - \Delta_{\{t_1\}} \dots \Delta_{\{t_l\}} \theta &(M \setminus L) =  
      \sum_{I \subset L} (-1)^{|I|+1} \theta((M \setminus L) \cup I)\\
      & \quad \forall \, M \in \finite(T)\setminus\{\emptyset\},\, \emptyset \neq  L = \{t_1,\dots,t_l\} \subset M
    \end{align*}
    for model (\ref{model}) defines a max-stable random field $X^*$ on $T$ with standard Fr{\'e}chet marginals which realizes $\theta$ as its own extremal coefficient function $\theta^*$.
  \end{enumerate}
\end{theorem}

\begin{proof}
If $\theta$ is an extremal coefficient function of a max-stable random field $X$ on $T$, then necessarily  $\theta(\emptyset) = 0$ and $\theta(\{t\}) = 1$ for all $t \in T$ (cf. (\ref{theta_intro})). Because of (\ref{theta_lHopital}) we have
  \begin{align*}
    \theta(M) = \lim_{x \uparrow  \infty}
    \frac{\PP\left(\exists \, t \in M  \text{ such that } X_t \geq x \right)}{\PP\left(X_t \geq x \right)} 
    = \lim_{x \uparrow  \infty} \frac{C^{(x)}(M)}{p^{(x)}},
     \end{align*}
where $C^{(x)}$ denotes the capacity functional for the binary random field $Y_t=\Eins_{X_t \geq x}$ and $p^{(x)} = \EE Y_t = 1-e^{-1/x}$. Since complete alternation respects scaling and pointwise limits, complete alternation of $\theta$ follows from theorem \ref{cap_CA}. This shows the necessity of (i),(ii),(iii).

Now, let $\theta:\finite(T)\rightarrow \RR$ be a function satisfying conditions (i),(ii),(iii) and let $(\tau^M_L)$ as given above.  We check that the numbers $(\tau^M_L)$ fulfill the (in)equalities (\ref{tau1}),(\ref{tau2}),(\ref{tau3}). Indeed we have:
\begin{itemize}
\item[(\ref{tau1})]
The inequalities $\tau^M_L = - \Delta_{\{t_1\}} \dots \Delta_{\{t_l\}} \theta (M \setminus L) \geq 0$ follow directly from the complete alternation of $\theta$.
\item[(\ref{tau3})] For $t \in T$ we have $\tau^{\{t\}}_{\{t\}} = \theta_{\{t\}} = 1$.
\item[(\ref{tau2})] From the definition of $\Delta_{\{t\}}$ we observe
\begin{align*}
\tau^{M \cup \{t\}}_{L \cup \{t\}} 
&= - \Delta_{\{t\}}\Delta_{\{t_1\}} \dots \Delta_{\{t_l\}}\theta \big((M \cup \{t\}) \setminus (L \cup \{t\})\big) \\
&= - \Delta_{\{t_1\}} \dots \Delta_{\{t_l\}}\theta \big(M \setminus L\big) + \Delta_{\{t_1\}} \dots \Delta_{\{t_l\}}\theta \big(M \cup \{t\} \setminus L\big) \\
&=  \tau^M_L -\tau^{M \cup \{t\}}_L.
\end{align*}
\end{itemize}
Thus, $(\tau^M_L)$ define a max-stable random field $X^*$ on $T$ with standard Fr{\'e}chet marginals as given by model (\ref{model}). Finally, we compute the extremal coefficient function $\theta^*$ of $X^*$ and see that it coincides with $\theta$:
\begin{align*}
  \theta^*(M) & \stackrel{(\ref{model})}{=} \sum_{\emptyset \neq L \subset M} \tau^M_L
= \sum_{\emptyset \neq L \subset M}  \sum_{\substack{I \subset L}} (-1)^{|I|+1} \theta((M \setminus L) \cup I) \\
&= \sum_{\emptyset \neq K \subset M} \theta(K)  \sum_{\substack{\emptyset \neq L \subset M\\ M \setminus L \subset K}} (-1)^{|K \cap L|+1} 
= \sum_{\emptyset \neq K \subset M} \theta(K)  \left(-\left(-\eins_{K=M}\right)\right)
= \theta(M). \qedhere 
\end{align*}
\end{proof}

So, theorem  \ref{ecf_CA} is a complete characterization of all possible extremal coefficient functions $\theta: \finite(T) \rightarrow [0,\infty)$ of max-stable random fields, and it provides a max-stable random field $X^*$ corresponding to a given valid extremal coefficient function.
  In \cite{schlathertawn_02} the last issue of the proof is derived for finite sets $T$ by a Moebius inversion. The relation to their proof becomes more transparent if we  compute $\theta^*(A)$ for $A \subset M$ from $\left(\tau^M_L\right)_{\emptyset \neq L \subset M}$ directly:
  \begin{align}\label{thetaA.from.tauML}
    \theta^*(A) \stackrel{(\ref{model})}{=} \sum_{\emptyset \neq K \subset A}\tau^A_K 
    \stackrel{(\ref{tau4})}{=} \sum_{\emptyset \neq K \subset A} \, \sum_{J \subset M \setminus A} \tau^M_{K \cup J} = \sum_{L \subset M \pmid L \cap A \neq \emptyset} \tau^M_L
  \end{align}
  Further, it follows that the bivariate distribution of the random field $X^*$ is given by
  \begin{align}\label{modelbivariate}
    - \log \PP(X^*_s \leq x, X^*_t \leq y) & = \eta(s,t)\min\left(\frac{1}{x},\frac{1}{y}\right)+ \max\left(\frac{1}{x},\frac{1}{y}\right),
  \end{align}
  where $\eta=1-\chi$  (cf. (\ref{eta_intro})). We will use this later.

\subsection{Consequences of complete alternation of extremal coefficients}

Here, we collect some immediate consequences of theorem \ref{ecf_CA} that concern the closure of extremal coefficient functions $\theta$ and extremal correlation functions $\chi$ of max-stable random fields, the operation of Bernstein functions on such functions, and a reinterpretation of the role of coefficients $\tau^M_L$ from model $(\ref{model})$ in terms of an integral representation of $\theta$.

\paragraph{Closure} Since $\theta$ may be characterized by a list of (in)equalities as given by theorem \ref{ecf_CA}, we may take pointwise limits (which may be read as ``setwise'' limits since the points of $\finite(T)$ are in fact sets):

\begin{corollary}\label{ecf_closed} The set of extremal coefficient functions $\theta: \finite(T) \rightarrow [0,\infty)$ of max-stable random fields on a given space $T$ is closed under pointwise convergence.
\end{corollary}

\begin{corollary}\label{countable_corr_closed}
If $T$ is countable, the set of extremal correlation functions $\chi: T \times T \rightarrow [0,1]$  of max-stable random fields on a given space $T$ (cf.(\ref{chi_intro})) is closed under pointwise convergence. 
\end{corollary}

\begin{proof} Let $\chi_n: T \times T \rightarrow [1,2]$ be extremal correlation functions of max-stable random fields for $n \in \NN$, such that we have $\lim_{n \to \infty} \chi_n(s,t) \rightarrow f(s,t)$ for all $(s,t) \in T \times T$.
Let $\theta_n: \finite(T) \rightarrow [0,\infty)$ be the corresponding extremal coefficient functions for $n \in \NN$.  On each finite set $M$, the value $\theta_n(M)$ is bounded by $|M|$. Thus, each $\theta_n$ in fact defines a point in the following product space ${\mathcal P}$ (by evaluation):
\begin{align*}
\theta_n =   \left(\theta_n(M)\right)_{M \in \finite(T)} \in \{0\} \times \prod_{M \in \finite(T)\setminus\{\emptyset \}} [1,|M|]=:{\mathcal P}
\end{align*}
The topology of pointwise convergence on $\finite(T)$ corresponds to the product topology on ${\mathcal P}\subset[0,\infty)^{\finite(T)}$. Since the intervals $[1,|M|]$ are compact and the product is taken over the countable set $\finite(T)$ (since $T$ is countable), the sequence of extremal coefficient functions $\theta_n: \finite(T) \rightarrow [0,\infty)$ possesses a convergent subsequence $\theta_{n'} \rightarrow \theta$ with respect to pointwise convergence. Corollary \ref{ecf_closed} now tells us that $\theta$ is again the extremal coefficient function of a max-stable random field on $T$. The observation that $f(s,t)=\lim_{n' \to \infty} \chi_{n'}(s,t)=\lim_{n' \to \infty} 2-\theta(\{s,t\}) = 2-\theta(\{s,t\})$ shows that the limit function $f$ on $T \times T$ is again an extremal correlation function of a max-stable random field. 
\end{proof}

\paragraph{Operation of Bernstein functions} 
Since condition (iii) in theorem \ref{ecf_CA} is equivalent to $\theta$ being negative definite, we get that Bernstein functions operate on extremal coefficient functions. 

\begin{definition} (Bernstein function) \\
A function $g:[0,\infty) \rightarrow [0, \infty)$  is called a {\em Bernstein function}  if one of the following equivalent conditions is satisfied (cf. \cite{bcr_84}, 4.4.3 and p. 141)
\begin{enumerate}[(i)]
\item The function $g$ is of the following form  for some $c,b \geq 0$ and $\nu$ a positive Radon measure on $(0,\infty)$, such that  $\int_0^\infty \frac{\lambda}{1+\lambda} \nu(\D \lambda) < \infty$: 
  \begin{align*}
    g(r)= c + br + \int_0^\infty \left(1-e^{-\lambda r }\right) \nu(\D \lambda)
  \end{align*}

\item $g$ is continuous and $g \in C^\infty((0,\infty))$ with $g \geq 0$ and $(-1)^n g^{(n+1)} \geq 0$ for all $n\geq 0$. (Here, $g^{(n)}$ denotes the $n$-th derivative of $g$.) 
\item $g$ is continuous, $g \geq 0$ and $g$ is negative definite as a function on the semigroup $([0,\infty),+,0)$.
\end{enumerate}
\end{definition}

For a comprehensive treatise on Bernstein functions including a table of examples see \cite{schillingetalii_10}.  
Since Bernstein functions operate on negative definite kernels (cf. \cite{bcr_84}, 3.2.9 and 4.4.3), we have the following corollary:

\begin{corollary} \label{theta_ops}
Let $\theta: \finite(T) \rightarrow \RR$ be an extremal coefficient function of a max-stable random field and $g$ a Bernstein function which is not constant. Then the function 
\[
M\, \mapsto 
\frac{g(\theta(M))-g(0)}{g(1)-g(0)}
\] 
is again an extremal coefficient function of a max-stable random field.  
\end{corollary}

For instance, if $\theta$ is a valid extremal coefficient function of a max-stable random field , then also $\log(1+\theta)/\log(2)$ or $\theta^q$ for $0<q<1$ are valid.
Of course, this leads also to the corresponding transformations on the set of bivariate quantities  $\chi$ and $\eta$ (cf. (\ref{chi_intro}),(\ref{eta_intro})). 

Further, the characterization of theorem \ref{ecf_CA} allows to derive many simple inequalities that an extremal coefficient function $\theta$ of a max-stable random field necessarily satisfies:

\begin{corollary}\label{thetaminusone_inequality}
Let  $\theta: \finite \rightarrow \RR$ be an extremal coefficient function of a max-stable random field and $g$ a Bernstein function.
Then the function $g(\theta-1): \finite\setminus\{\emptyset\} \rightarrow \RR $ satisfies the triangle inequality
\begin{align*}
g\left(\theta(A \cup B) - 1\right) \leq g\left(\theta(A \cup C) - 1\right) + g\left(\theta(C \cup B) - 1\right) &&& \text{for $A,B,C \in \finite(T)\setminus\{ \emptyset \}$.}
\end{align*}
\end{corollary}

\begin{proof}
Since $\theta$ is an extremal coefficient function of a max-stable random field, it is completely alternating which is equivalent to being negative definite (in the semigroup sense, cf. theorem \ref{ecf_CA}). Subtracting $1$ does not change this property. But notice further that $\theta$ takes values in $\{0\} \cup [1,\infty)$, where the value $0$ is only attained for the empty set $\emptyset$. Thus, $\theta-1: \finite \setminus \{\emptyset\} \rightarrow \RR$ is negative definite and takes values only in $[0,\infty)$. Applying a Bernstein function $g$ does not change this property. By \cite{bcr_84} 8.2.7, this also means that $f:=g(\theta-1): \finite\setminus\{\emptyset\} \rightarrow \RR $ is completely alternating on $\finite\setminus\{\emptyset\}$. Since we have also $f \geq 0$ on $\finite\setminus\{\emptyset\}$, we may derive for $A,B,C \in \finite \setminus \{\emptyset\}$ 
\begin{align*}
&f(A \cup B)-f(A \cup C)-f(C \cup B)\\
&\leq f(C)+f(A \cup B)-f(A \cup C)-f(C \cup B)\\
&=(f(C)-f(A \cup C)-f(C \cup B)+f(A \cup B \cup C))+(f(A \cup B)-f(A \cup B \cup C))\\
&=\Delta_{A}\Delta_{B}f(C) + \Delta_{C}f(A \cup B) \leq 0. \qedhere
\end{align*}
\end{proof}

By letting $A=\{s\}$, $B=\{t\}$, $C=\{r\}$  in  corollary \ref{thetaminusone_inequality}, we may also derive such inequalities for the respective bivariate quantities:

\begin{corollary}\label{eta_inequality}
The function $\eta=1-\chi$ of a max-stable random field on $T$ (cf. (\ref{eta_intro})) satisfies for any Bernstein function $g$ the triangle inequality
\begin{align*}
g \circ \eta(s,t) \leq g \circ \eta(s,r) + g \circ \eta(r,t) &&& \text{for $s,t,r \in T$.}
\end{align*}
\end{corollary}

\paragraph{Integral representation} 
Theorem 2 also allows to reinterpret the role of the coefficients $\tau^M_L$ of the max-linear model (\ref{model}) within the framework of integral representations of negative definite functions.

\begin{definition}
  A {\em bounded semicharacter} $\rho$ on $\finite(T)$ is a function $\rho: \finite(T) \rightarrow \{0,1\}$ that satisfies $\rho(\emptyset)=1$ and $\rho(A \cup B) = \rho(A)\rho(B)$ for $A,B \in \finite(T)$. 

The set of bounded semicharacters on $\finite(T)$ is denoted $\widehat{\finite(T)}$ and forms itself a semigroup with respect to pointwise multiplication with neutral element the constant function $1$.
\end{definition}

Since we have theorem \ref{ecf_CA}, we may apply \cite{bcr_84} 4.1.16, that is
any extremal coefficient function $\theta$ of a max-stable random field allows for an integral representation: 
\begin{corollary} \label{intrep}
Let $\theta: \finite(T) \rightarrow \RR$ be an extremal coefficient function of a max-stable random field. Then $\theta$ uniquely determines a positive Radon measure $\mu$ on $\widehat{\finite(T)}\setminus\{1\}$ such that 
\begin{align*} 
\theta(A) =  \mu(\{\rho \in \widehat{\finite(T)} \setminus \{1\} \mid \rho(A)=0 \})
\end{align*}
where $\theta(\{t\})=1$ for $t \in T$.
The function $\theta$ is bounded if and only if $\mu(\widehat{\finite(T)} \setminus \{1\})  <  \infty$.
\end{corollary}

If we restrict ourselves to a finite non-empty subset $M \subset T$, the restricted extremal coefficient function $\theta_M: \finite(M) \rightarrow [0,|M|]$ is still completely alternating and corollary \ref{intrep} applies, of course. But now, since we consider only a finite $M$, we can use the  additional isomorphism of finite (topologically discrete) idempotent abelian semigroups:
\begin{align*}
  (\widehat{\finite(M)},\cdot,1) \cong (\finite(M),\cup,\emptyset) \qquad \rho \mapsto \bigcap_{L \in \rho^{-1}(1)} (M \setminus L) 
\end{align*} 
in order to reformulate corollary \ref{intrep} for this situation:
Any such restricted extremal coefficient function $\theta_M: \finite(M) \rightarrow \RR$ uniquely determines a positive Radon measure $\mu_M$ on $\finite(M)\setminus\{\emptyset\}$ such that 
\begin{align*} 
\theta_M(A) =  \mu_M(\{ L \in \finite \setminus \{\emptyset\} \mid A \cap L \neq \emptyset \}) = \sum_{L \subset M \pmid A \cap L \neq \emptyset} \mu(\{L\})
\end{align*}
where $\sum_{t \in L} \mu_M(\{L\})=1$  for $t \in T$. 
A comparison with (\ref{thetaA.from.tauML}) reveals that 
\begin{align*}
\mu_M(\{L\})=\tau^M_L.
\end{align*}

\section{Convex geometry of extremal coefficients}\label{convexgeometry}
We will fix again a finite non-empty subset $M \subset T$. Remember that we denoted by $\DepSet_M$ the {\em dependency set} of $\{X_t\}_{t \in M}$ (cf. (\ref{depSet}))  which was introduced in \cite{molchanov_08} and which is the unique compact convex set $\DepSet_M \subset [0,\infty)^M$ satisfying 
\begin{align*}
  \ell_M(x)=\sup\{ \langle x,y \rangle \pmid y \in \DepSet_M\} 
  \qquad \forall \, 0\neq x \in [0,\infty)^M.
\end{align*}
Here, we will show that the dependency sets $\DepSet^*_M$ of our max-linear model (\ref{model}) play an exceptional role among all dependency sets sharing the same extremal coefficients:

\begin{theorem}\label{KMincludedinHalfspaceIntersection}
Consider a max-stable random field with extremal coefficient function  $\theta$ and denote $\DepSet_M$ its dependency set for a finite non-empty $M \subset T$. Further denote 
\begin{align*} \Half_L:=\{ x \in [0,\infty)^M \pmid \langle x , \eins^M_L\rangle \leq \theta(L) \} 
\end{align*} 
the half space in $[0,\infty)^M$ bounded by the plane with defining equation $\langle x , \eins^M_L\rangle = \theta(L)$.
\begin{enumerate}[a)]
\item\label{KMinHSIa}
Then $\DepSet_M \subset \bigcap_{\emptyset \neq L \subset M} \Half_L$.
\item\label{KMinHSIb} For each $\emptyset \neq L \subset M$ there exists a point $x^L$ in the intersection $x^L \in \DepSet_M \cap \Half_L$.
\item \label{KMinHSIc} The dependency set $\DepSet^*_M$ of $(X^*_t)_{t \in M}$ from model (\ref{model}) with extremal coefficient function  $\theta$ is given by $\DepSet^*_M = \bigcap_{\emptyset \neq L \subset M} \Half_L$.
\end{enumerate}
\end{theorem}
\begin{proof}
\begin{enumerate}[a)]
\item Let $\emptyset \neq L \subset M$  and $x \in \DepSet_M$. Assuming $\langle x,\eins^M_L \rangle > \theta(L)$ we would have the contradiction $\theta(L)=\ell_M(\eins^M_L)=\sup\{\langle x,\eins^M_L \rangle \pmid x \in \DepSet_M\} > \theta(L)$. So $\langle x,\eins^M_L \rangle \leq \theta(L)$ and part \ref{KMinHSIa}) is proven.
\item  Since $\DepSet_M$ is compact, we have  $\theta(L)=\ell_M(\eins^M_L)=\sup\{\langle x,\eins^M_L \rangle \pmid x \in \DepSet_M\} = \langle x(L),\eins^M_L \rangle$ for some $x^L \in \DepSet_M$. This shows part \ref{KMinHSIb}).
\item
Denote the right-hand side  by 
\begin{align*}
\PotentialDepSet_M := \bigcap_{\emptyset \neq L \subset M} \Half_L = \{ x \in [0,\infty)^M \pmid \langle x , \eins^M_L\rangle \leq \theta(L) \text{ for all } \emptyset \neq L \subset M\}.
\end{align*}
The inclusion $\DepSet^*_M \subset \PotentialDepSet_M$ is proven in part \ref{KMinHSIa}). So, we still need to show the other inclusion $\PotentialDepSet_M \subset \DepSet^*_M$. By definition (cf. (\ref{depSet})), $\DepSet^*_M \subset [0,1]^M$ is the unique compact convex set satisfying $ \ell_{M}(x)=\sup\{ \langle x,y \rangle \pmid y \in K^*_{M}\}$ for $0 \neq x \in [0,\infty)^M$.  
The closed convex set $\DepSet^*_M$ may also be described as the following intersection of half spaces (cf. \cite{schneider_93}, section 1.7.):
\begin{align*}
\DepSet^*_M= \bigcap_{x \in S_M} \{ y \in [0,\infty)^M \pmid \langle x,y \rangle \leq \ell_M(x)\}.
\end{align*}
Thus, it suffices to show the following implication in order to prove $\PotentialDepSet_M \subset \DepSet^*_M$:
\begin{align*}
x \in S_M \text{ and } y \in \PotentialDepSet_M \quad \Longrightarrow \quad 
\langle x,y \rangle \leq \ell_M(x) 
\end{align*}
We now prove this implication: Without loss of generality, we may label the elements of $M=\{t_1, \dots, t_m\}$ such that $x_{t_1} \geq x_{t_2} \geq \dots \geq x_{t_m}$.  Then we may write $x \in S_M \subset [0,\infty)^M$ as
\begin{align*}
x = \underbrace{x_{t_m}}_{\geq 0} \eins^M_M + \underbrace{(x_{t_{n-1}}-x_{t_m})}_{\geq 0} \eins^M_{M \setminus \{t_m\}} + \dots + \underbrace{(x_{t_{2}}-x_{t_3})}_{\geq 0} \eins^M_{\{t_1,t_2\}} + \underbrace{(x_{t_{1}}-x_{t_2})}_{\geq 0} \eins^M_{\{t_1\}}
\end{align*}
Taking the scalar product with $y \in \PotentialDepSet_M$, we conclude 
\begin{align}
\notag \langle x,y \rangle 
& \leq x_{t_m} \theta(M) + (x_{t_{n-1}}-x_{t_m}) \theta(M \setminus \{t_m\}) + \dots \\ \notag & \hspace{4cm} \dots + (x_{t_{2}}-x_{t_3}) \theta(\{t_1,t_2\}) + (x_{t_{1}}-x_{t_2}) \theta(\{t_1\})\\
\label{ellcoincide} & = x_{t_m} (\theta(M)-\theta(M \setminus \{t_m\}))  + \dots + x_{t_{2}} (\theta(\{t_1,t_2\})-\theta(\{t_1\})) + x_{t_{1}} \theta(\{t_1\})
\end{align}
On the other hand the stable tail dependence function $\ell_M$ of model (\ref{model}) is by this ordering of the components of $x$ given as 
\begin{align*}
\ell_M(x) &= \sum_{\emptyset \neq L \subset M} \tau^M_L \max_{t \in L} x_{t} = \sum_{i=1}^m x_{t_i} \left( \sum_{L \subset M \pmid t_1,\dots,t_{i-1} \notin L, t_{i} \in L} \tau^M_L \right)
\end{align*}
From formula (\ref{thetaA.from.tauML}) we see that this expression of $\ell_M(x)$ coincides with the right-hand side of $(\ref{ellcoincide})$. Thus, we have our desired inequality $\langle x,y \rangle \leq \ell_M(x)$. This shows part \ref{KMinHSIc}). \qedhere
\end{enumerate}
\end{proof}

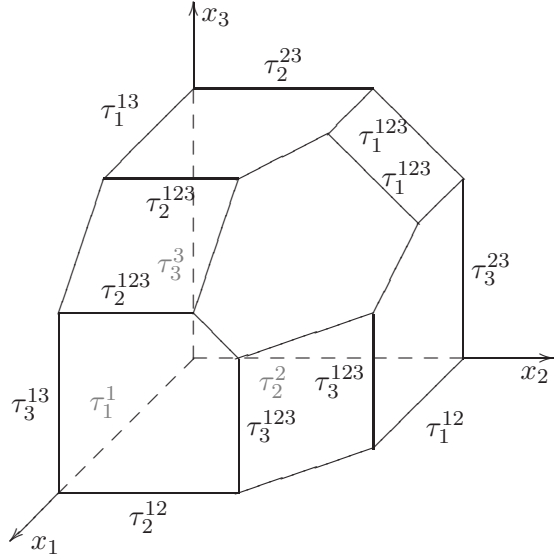
\begin{figure}[htb]
\centerline{
  \xymatrix@=17pt@*=0{
&&&& \ar@{<-}^(.2){\txt{$x_3$}}[dd] 
  \\ \\
&&&& {\punkt} \ar@{-}^-{\txt{$\tau^{23}_2$}}[rrrr] &&&& {\punkt} 
\\
&&&&&&& {\punkt} \ar@{-}_-{\txt{$\tau^{123}_1$}}[ur] 
\\
&& {\punkt} \ar@{-}^-{\txt{$\tau^{13}_1$}}[uurr] \ar@{-}_-{\txt{$\tau^{123}_2$}}[rrr] &&& {\punkt} \ar@{-}[urr] 
&&&&& {\punkt} \ar@{-}[uull]
\\
&&&&&&&&& {\punkt} \ar@{-}[uull] \ar@{-}^-{\txt{$\tau^{123}_1$}}[ur]
\\ \\
& {\punkt} \ar@{-}[uuur] \ar@{-}^-{\txt{$\tau^{123}_2$}}[rrr] &&& {\punkt} \ar@{-}[uuur] 
&&&& {\punkt} \ar@{-}[uur] 
\\
&&&& {\punkt} \ar@{--}^<<<<<<<<<<<<<{\txt{\textcolor{taucol}{$\tau^{3}_3$}}}[uuuuuu] \ar@{--}_<<<<<<<<<<<{\txt{\textcolor{taucol}{$\tau^{2}_2$}}}[rrrrrr] \ar@{--}_-{\txt{\textcolor{taucol}{$\tau^{1}_1$}}}[dddlll]  
& {\punkt} \ar@{-}[ul] \ar@{-}[urrr]  &&&&& {\punkt} 
\ar@{-}_-{\txt{$\tau^{23}_3$}}[uuuu] && \ar@{<-}^(.2){\txt{$x_2$}}[ll] 
\\ \\
&&&&&&&& {\punkt} \ar@{-}^-{\txt{$\tau^{123}_3$}}[uuu] \ar@{-}_-{\txt{$\tau^{12}_1$}}[uurr]
\\ 
& {\punkt} \ar@{-}^-{\txt{$\tau^{13}_3$}}[uuuu] \ar@{-}_-{\txt{$\tau^{12}_2$}}[rrrr] 
&&&& {\punkt} \ar@{-}_-{\txt{$\tau^{123}_3$}}[uuu] \ar@{-}[urrr] 
\\
\ar@{<-}_(.2){\txt{$x_1$}}[ur]
}
}
\caption{Dependency set $\DepSet^*_M$ of $\{X^*_{t}\}_{t \in M}$ of model (\ref{model}) for $M = \{1,2,3\}$.}
\label{dependencysetfigure}
\end{figure}

So, if we fix the extremal coefficient function $\theta$ of a max-stable random field on $T$, then model (\ref{model}) yields a maximal dependency set $\DepSet^*_M$ w.r.t. inclusion, that is
\begin{align*}
\DepSet^*_M = \bigcup_{\substack{\DepSet_M \text{ dependency set}\\ \text{with the same extremal coefficients as } \DepSet^*_M}} \DepSet_M. 
\end{align*}
It is an open problem and it would be interesting to know whether there exist also minimal dependency sets and if they would help to better understand the classification of all dependency structures. A very naive idea would be to take for each $L$ in each of the sets $\{ x \in [0,\infty)^M \pmid \langle x , \eins^M_A \rangle = \theta(A) \} \cap \bigcap_{L \subset M, L \neq \emptyset} \Half_L$ for $\emptyset \neq A \subset M$ one point and then to take the convex hull with 0 included. However this fails to be a dependency set in dimensions $|M| \geq 3$, since it is not even a zonoid, which would be necessary (cf. \cite{molchanov_08}).

\section{Continuity}\label{continuity}

\paragraph{Continuity of functions on $\finite(T)$} In this section we require $T$ to be a metric space. We technically prepare the notion of continuity that we will use in connection with extremal coefficient functions $\theta: \finite(T) \rightarrow [0,\infty)$. Therefore let $f: \finite(T) \rightarrow \RR$ be a function on the finite subsets of $T$. Then $f$ induces a familiy of functions $\{f^{(m)}\}_{m \geq 0}$ where $f^{(m)}: T^m \rightarrow \RR$ is given by \begin{align}\label{fcorrespondency}
f^{(m)}(t_1,\dots,t_m)=f(\{t_1,\dots,t_m\}).
\end{align}
We see that for each $m$ the induced function $f^{(m)}:T^m\rightarrow \RR$ is symmetric, that is for any permutation $\pi \in S_m$ 
\begin{align}\label{fsymmetry}
f^{(m)}(t_1,\dots,t_m)=f^{(m)}(t_{\pi(1)},\dots,t_{\pi(m)}).
\end{align}
Further, the family $f^{(m)}:T^m\rightarrow \RR$ is consistent in the following way: For any non-negative integers $k_1,\dots,k_r$ with $\sum_{i=1}^r k_i=m$ and $(t_1,\dots,t_r) \in T^r$ we have
\begin{align}\label{fconsistency}
f^{(m)}(\underbrace{t_1,\dots,t_1}_{k_1 \text{ times}},\dots,\underbrace{t_r,\dots,t_r}_{k_r \text{ times}}) = f^{(r)}(t_1,\dots,t_r).
\end{align}
On the other hand, given functions $\{f^{(m)}\}_{m \geq 0}$ which satisfy (\ref{fsymmetry}) and (\ref{fconsistency}), they define a well-defined function $f: \finite(T) \rightarrow \RR$ via (\ref{fcorrespondency}). 

\begin{definition} 
Let $f: \finite(T) \rightarrow \RR$ be a function on the finite subsets of a metric space $T$. We say that $f$ is {\em continuous} if all induced functions $f^{(m)}:T^m \rightarrow \RR$ are continuous for all $m \geq 0$.
\end{definition}

\begin{remark}
To be more precise, this notion of continuity for functions on $\finite(T)$ used in this definition corresponds to the final topology  on $\finite(T)$ when we view it as the direct limit $\finite(T)=\varinjlim \finite^m(T)$ where $\finite^m(T)$ denotes the set of finite subsets of $T$ with number of elements less than or equal to $m$, and where $\finite^m(T)=T^m/\sim$ carries the quotient topology of $T^m$.
\end{remark}

\paragraph{Continuity of $\theta$ and $\chi$}

For an extremal coefficient function $\theta: \finite(T) \rightarrow \RR$  we denote its induced functions as described in the paragraph above by $\theta^{(m)}$. 
\begin{theorem}\label{cty_chitheta}
Let $X$ be a max-stable random field on a metric space $T$ with standard Fr{\'e}chet marginals, $\theta$ its extremal coefficient function and $\chi$ its extremal correlation function. Further, let $X^*$ be the random field from model (\ref{model}) that realizes $\theta$ as its extremal coefficient function, and thus also $\chi$ as its extremal correlation function (cf. theorem \ref{ecf_CA}). 
Then the following implications hold:

\begin{enumerate}[a)]
\item $X$ is stochastically continuous.  $\quad \Longrightarrow \quad \theta$  is continuous. $\quad \Longrightarrow \quad \chi$ is continuous.
\item \label{cty_ineq} The function $\eta=1-\chi$ (cf. (\ref{eta_intro})) satisfies for any $\varepsilon > 0$ 
\[
\PP(|X^*_s - X^*_t| > \varepsilon) \leq 2 \left( 1 - \exp\left( - \frac{\eta(s,t)}{\varepsilon} \right)\right) \leq \frac{2}{\varepsilon} \, \eta(s,t).
\]
\item $\chi$ is continuous. $\quad \Longrightarrow \quad$  $X^*$ is stochastically continuous.
\end{enumerate}
\end{theorem}

\begin{proof}
\begin{enumerate}[a)]
\item The second implication is obvious from $\chi(s,t)=2-\theta^{(2)}(s,t)$. As to the first implication: Stochastic continuity of $X$ means that for any $\varepsilon > 0$, for any $t \in T$ and sequence $t^{(n)} \rightarrow t$ we have $\PP(\lvert X_{t^{(n)}} - X_t \rvert > \varepsilon) \rightarrow 0$. From this, we can easily derive that for any $\varepsilon >0$ and any $(t_1,\dots,t_m) \in T^m$ and a sequence  $(t^{(n)}_1,\dots,t^{(n)}_m) \rightarrow (t_1,\dots,t_m)$, also $\PP(\lVert (Z_{t^{(n)}_i} - Z_{t_i})_{i=1}^m \rVert > \varepsilon) \rightarrow 0$ for any reference norm $\lVert \cdot \rVert$ on $\RR^m$. The latter implies the corresponding convergence in distribution: $F_{(t^{(n)}_1,\dots,t^{(n)}_m)} \rightarrow F_{(t_1,\dots,t_m)}$. Since $\log F_{(t_1, \dots, t_m)}: [0,\infty)^m \rightarrow \RR$ is monotone and homogeneous, we have that for $x>0$ the point $(x,\dots,x) \in (0,\infty)^m$ is a continuity point of  $F_{(t_1, \dots, t_m)}$ (cf. \cite{resnick_08}, p. 277). Thus, we may conclude that the induced function $\theta^{(m)}$ on $T^m$ is continuous, since $\theta^{(m)}(t_1,\dots,t_m)= - x \log F_{(t_1, \dots, t_m)} (x, \dots, x)$.
 
\item  Let $\varepsilon > 0$. We will prove the statement for $2 \varepsilon$ instead of $\varepsilon$. Therefore, consider the following disjoint events on a corresponding probability space $(\Omega, {\mathcal A}, \PP)$ for $k=0,1,2, \dots $
\begin{align*}
A_k := \left\{\omega  \in \Omega \pmid (X^*_s(\omega),X^*_t(\omega)) \in (k \varepsilon, (k+2) \varepsilon]^2 \setminus ((k+1)\varepsilon, (k+2)\varepsilon]^2 \right\} 
\end{align*}
The disjoint union $\bigcup_{k=0}^{\infty} A_k$ is a subset of $\{\omega \in \Omega \pmid |X^*_s(\omega) - X^*_t(\omega)| \leq 2\varepsilon\}$ and so 
\begin{align*}
\PP(|X^*_s - X^*_t| \leq 2\varepsilon) \geq \PP\left(\bigcup_{k=0}^{\infty} A_k\right) = \sum_{k=0}^\infty \PP(A_k) = \lim_{n \to \infty} \sum_{k=0}^n \PP(A_k).
\end{align*}
For further calculations we abbreviate for $p,q \in \NN \cup \{0\}$ 
\begin{align*}
B(p,q):=\PP(X^*_s \leq p \cdot \varepsilon, X^*_t \leq q \cdot \varepsilon).
\end{align*}
Note from (\ref{modelbivariate}) that $B(p,q)=B(q,p)$ and $B(p,0)=0$. With this notation we rearrange
\begin{align*}
\sum_{k=0}^n \PP(A_k) 
&= -B(n+1,n+1)   + 2  \sum_{k=0}^n \left[ B(k+2,k+1)-  B(k+2,k) \right]
\end{align*}
For the second summand we have (cf. (\ref{modelbivariate}))
\begin{align*}
&   \sum_{k=0}^n \left[ B(k+2,k+1)-  B(k+2,k) \right]\\
&\stackrel{(\ref{modelbivariate})}{=}    \sum_{k=0}^n \left[ \exp\left(-\frac{1}{\varepsilon}\left[\frac{\eta(s,t)}{k+2} + \frac{1}{k+1}\right]\right) - \exp\left(-\frac{1}{\varepsilon}\left[\frac{\eta(s,t)}{k+2} + \frac{1}{k}\right]\right)  \right]\\
&=    \sum_{k=0}^n \exp\left(-\frac{1}{\varepsilon}\left[\frac{\eta(s,t)}{k+2} \right]\right) \left[ \exp\left(-\frac{1}{(k+1)\varepsilon}\right) - \exp\left(-\frac{1}{k \varepsilon}\right)  \right]\\
& \geq    \sum_{k=0}^n \exp\left(-\frac{\eta(s,t)}{2 \varepsilon}\right) \left[ \exp\left(-\frac{1}{(k+1)\varepsilon}\right) - \exp\left(-\frac{1}{k \varepsilon}\right)  \right] \\
&=  \exp\left(-\frac{\eta(s,t)}{2 \varepsilon}\right) \exp\left(-\frac{1}{(n+1) \varepsilon}\right) 
\end{align*}
Finally, 
\begin{align*}
\PP(|X^*_s &- X^*_t| > 2\varepsilon)  = 1 - \PP(|X^*_s - X^*_t| \leq 2\varepsilon)  \leq 1 - \lim_{n \to \infty} \sum_{k=0}^n \PP(A_k)\\
&= 1 + \lim_{n \to \infty} B(n+1,n+1)  - 2  \lim_{n \to \infty} \sum_{k=0}^n \left[ B(k+2,k+1)-  B(k+2,k) \right] \\
&\leq 1 + \lim_{n \to \infty} \exp\left(-\frac{\eta(s,t)+1}{(n+1)\varepsilon}\right) - 2  \lim_{n \to \infty} \left(\exp\left(-\frac{\eta(s,t)}{2 \varepsilon}\right) \exp\left(-\frac{1}{(n+1) \varepsilon}\right)\right)\\
&= 2 - 2 \exp\left(-\frac{\eta(s,t)}{2 \varepsilon}\right) \leq \frac{2}{2\varepsilon} \eta(s,t)
\end{align*}
\item This follows from part \ref{cty_ineq}). \qedhere
\end{enumerate}
\end{proof}

The following corollaries are immediate consequences of theorem \ref{cty_chitheta}.

\begin{corollary}\label{chitheta_cty}
Let $\theta: \finite(T) \rightarrow [0,\infty)$ be the extremal coefficient function and $\chi:T\times T \rightarrow [0,1]$ the extremal correlation function of a max-stable random field on a metric space $T$ with standard Fr{\'e}chet marginals. Then
\begin{align*}
\theta \text{ is continuous.} \quad \Longleftrightarrow \quad \chi \text{ is continuous.} 
\end{align*}
\end{corollary}

\begin{corollary}
Let $T$ be a  metric space.
\begin{enumerate}[a)]
\item
The set of {\em continuous} extremal coefficient functions $\theta:\finite(T) \rightarrow [0,\infty)$ of max-stable random fields on $T$ coincides with the set of extremal coefficient functions $\theta:\finite(T) \rightarrow [0,\infty)$ of {\em stochastically continuous} max-stable random fields on $T$.
\item
The set of {\em continuous} extremal correlation functions $\chi: T \times T \rightarrow [0,1]$ of max-stable random fields on $T$ coincides with the set of extremal correlation functions $\chi: T \times T \rightarrow [0,1]$ of {\em stochastically continuous} max-stable random fields on $T$.
\end{enumerate}
\end{corollary}

Finally, we will show that we are allowed to take pointwise limits on the set of extremal correlation functions of max-stable random fields as long as the limit function is continuous:

\begin{theorem}\label{corr_cvgce_to_cts}
Let $T$ be a separable metric space and let $\chi_n:T \times T \rightarrow [0,1]$ be the extremal correlation functions of some max-stable random fields on $T$ such that the functions $\chi_n$ converge pointwise to a continuous function $f: T \times T \rightarrow [0,1]$. Then $f$ is an extremal correlation function of a stochastically continuous, max-stable random field on $T$.
\end{theorem}

\begin{remark}
However, we are not aware of whether it is still true, that  the limit function $f$ is an extremal correlation function of a max-stable random field on $T$ if we drop the requirement of $f$ being continuous.
\end{remark}

\begin{proof}[Proof of theorem \ref{corr_cvgce_to_cts}]  Let $Q \subset T$ be a countable dense set; then $Q^m \subset T^m$ is also countable and dense for any $m \geq 0$.
Since $\chi_n$ converge pointwise to $f$, the restrictions $\left.\chi_n\right|_{Q \times Q}$ also converge pointwise to $\left.f\right|_{Q \times Q}$. Because of corollary \ref{countable_corr_closed} and the arguments in the corresponding proof, the restricted function $\left.f\right|_{Q \times Q}$ is an extremal correlation function $\tilde{\chi}$ of a max-stable random field $\tilde{X}$ on $Q$, which is a random field of type (\ref{model}). Let us denote its extremal coefficient function by $\tilde{\theta}$. 

The idea of this proof is now to show that  the induced functions $\tilde{\theta}^{(m)} : Q^m \rightarrow [1,m]$ can be uniquely extended to continuous functions on $T^m$. For $m=0$ and $m=1$ this is trivial, since $\tilde{\theta}^{(0)} \equiv 0$ and $\tilde{\theta}^{(1)} \equiv 1$ are constant by definition.
Secondly, since $f$ is continuous on $T \times T$ the function $2-f:T\times T \rightarrow [1,2]$ is the unique continuous extension of $\tilde{\theta}^{(2)}=\left.2-f\right|_{Q \times Q}$ to $T \times T$.
 
Moreover, set $\eta=1-f$. Then $\eta$ is also continuous with $\eta(t,t)=0$ for all $t \in T$. Since $\eta(s,t)$ is the pointwise limit of functions which are symmetric and satisfy the triangle inequality (cf. corollary \ref{eta_inequality}), the function $\eta$ itself is symmetric and satisfies the triangle inequality $\eta(s,t)\leq \eta (s,r)+\eta(r,t)$ for all $s,t,r \in T$.

Now, fix $m \geq 3$ and consider $\tilde{\theta}^{(m)} : Q^m \rightarrow [1,m]$. Note that even though we can derive continuity of $\tilde{\theta}^{(m)}$ from corollary \ref{chitheta_cty}, we may not derive a continuous extension directly. 
But the following property of $\tilde{\theta}^{(m)}$ will be enough: Let $B_\delta(t)\subset T^m$ denote  the open $\delta$-ball around $t \in T^m$. \smallskip\\
{\bf Property A} For $t=(t_1,\dots,t_m) \in T^m$ and $\varepsilon > 0$ we can find $\delta>0$ such that for any $q, \hat{q} \in B_\delta(t) \cap Q^m$ we have $|\tilde{\theta}^{(m)}(\hat{q})-\tilde{\theta}^{(m)}(q)|\leq \varepsilon$.
\smallskip\\
This means, that $\tilde{\theta}^{(m)}:Q^m \subset T^m \rightarrow [1,m]$ is Cauchy continuous. In fact, because of the symmetry, it suffices to establish the following property:
\smallskip\\
{\bf Property B} For $t=(t_1,\dots,t_m) \in T^m$ and $\varepsilon > 0$ we can find $\delta>0$ such that for any $q, \hat{q} \in B_\delta(t) \cap Q^m$ we have $\tilde{\theta}^{(m)}(\hat{q})-\tilde{\theta}^{(m)}(q) \leq \varepsilon$.
\smallskip\\
Interchanging the roles of $q$ and $\hat{q}$ in property B yields property A with the smaller $\delta$ from property B. Now, we prove property B in three steps:

\paragraph{1st step} For all $\varepsilon_1>0$ there is $\delta_1>0$, such that for all $q, \hat{q} \in B_{\delta_1}(t) \cap Q^m$ we have $2 m \sum_{i=1}^m \eta(q_i,\hat{q}_i) < \varepsilon_1$.\medskip\\
{\em Proof of 1st step:} Since $s \mapsto \eta(t_i,s)$ is continuous at $t_i\in T$ for $i=1,\dots,m$, there exist $\delta_1^{(i)}>0$ such that $\eta(t_i,q_i) < \frac{\varepsilon_1}{4m^2}$ for $q_i \in B_{\delta^{(i)}_1}(t_i)$. Choose $\delta_1>0$ such that $B_{\delta_1}(t) \subset \prod_{i=1}^m B_{\delta^{(i)}_1}(t_i)$. Let $q,\hat{q} \in B_{\delta_1}(t)$. Then the triangle inequality of $\eta$ yields:
\[2 m \sum_{i=1}^m \eta(q_i,\hat{q}_i) \leq  2 m \sum_{i=1}^m \left[\eta(q_i,t_i) + \eta(t_i,\hat{q}_i)\right] < 2m \sum_{i=1}^m 2 \frac{\varepsilon_1}{4m^2} = \varepsilon_1
\]

\paragraph{2nd step} For $q \in Q^m$ denote $\tilde{F}_q(x):=\PP(\tilde{X}_{q_1} \leq x, \dots, \tilde{X}_{q_m} \leq x)$. For all $\varepsilon_1>0$ there is $\delta_1>0$, such that for all $q, \hat{q} \in B_{\delta_1}(t) \cap Q^m$ and for all $a,x>0$ we have $ \tilde{F}_q(x) < \tilde{F}_{\hat{q}}(x+a) + \frac{\varepsilon_1}{a}$.\medskip\\
{\em Proof of 2nd step:}
A priori we have for any $a,x>0$ and all $q,\hat{q} \in Q^m$
\begin{align*}
\tilde{F}_q(x) &= \PP(\tilde{X}_{q_1} \leq x, \dots, \tilde{X}_{q_m} \leq x, \sum_{i=1}^m|\tilde{X}_{q_i} - \tilde{X}_{\hat{q}_i}| \leq a) \\
& \qquad + \PP(\tilde{X}_{q_1} \leq x, \dots, \tilde{X}_{q_m} \leq x, \sum_{i=1}^m|\tilde{X}_{q_i} - \tilde{X}_{\hat{q}_i}| > a) \\
&= \PP(\tilde{X}_{q_1} \leq x, \dots, \tilde{X}_{q_m} \leq x, |\tilde{X}_{q_1} - \tilde{X}_{\hat{q}_1}| \leq a, \dots, |\tilde{X}_{q_m} - \tilde{X}_{\hat{q}_m}| \leq a) \\
& \qquad + \PP\left(|\tilde{X}_{q_1} - \tilde{X}_{\hat{q}_1}| > \frac{a}{m} \text{ or } \dots \text{ or } |\tilde{X}_{q_m} - \tilde{X}_{\hat{q}_m}| > \frac{a}{m} \right) \\
&\leq \tilde{F}_{\hat{q}}(x+a) + \sum_{i=1}^m \PP\left(|\tilde{X}_{q_i} - \tilde{X}_{\hat{q}_i}| > \frac{a}{m}\right) 
\leq  \tilde{F}_{\hat{q}}(x+a) + \frac{2m}{a} \sum_{i=1}^m \eta(q_i,\hat{q}_i)
\end{align*}
The last inequality follows from  theorem \ref{cty_chitheta} part \ref{cty_ineq}). Now, the 1st step gives a suitable choice for $\delta_1$ for given $\varepsilon_1$.

\paragraph{3rd step}  {\em Proof of property B:} \medskip\\
We may assume $\varepsilon < 2$. Choose $a > 0$ such that $2am<\varepsilon$. In particular this implies $e^{-m}\geq e^{-\frac{\varepsilon}{2a}}$ and we may choose $\varepsilon_1$ such that on the one hand $\varepsilon_1 < a (e^{-m} -  e^{-\frac{\varepsilon}{2a}})$ and on the other hand $\varepsilon_1 < ae^{-m}/(\frac{2}{\varepsilon}-1)$. This ensures:
\begin{align}\label{epshalf}
a \left(- \log \left(e^{-m}-\frac{\varepsilon_1}{a}\right)\right) \leq \frac{\varepsilon}{2} 
\qquad \text{ and } \qquad
\frac{1}{e^{-m}-\frac{\varepsilon_1}{a}} \cdot \frac{\varepsilon_1}{a} \leq \frac{\varepsilon}{2}
\end{align}
Finally, choose $\delta:=\delta_1$ as in the 2nd step (or equivalently 1st step) for this $\varepsilon_1$. Then we have for arbitrary $q,\hat{q} \in B_{\delta}(t)$ that
\begin{align*}
\tilde{\theta}^{(m)}(\hat{q}) - \tilde{\theta}^{(m)}({q}) 
&=  (1+a)\left(-\log \tilde{F}_{\hat{q}}(1+a)\right) + \log \tilde{F}_q(1) \\
&\leq  (1+a)\left(-\log  \left(\tilde{F}_{q}(1) - \frac{\varepsilon_1}{a}\right)\right) + \log \tilde{F}_q(1) \\
&=  \left(\log \tilde{F}_q(1) - \log  \left(\tilde{F}_{q}(1) - \frac{\varepsilon_1}{a}\right)\right) 
+ a \left(-\log  \left(\tilde{F}_{q}(1) - \frac{\varepsilon_1}{a}\right)\right).
\end{align*}
Here, the first equality is just the definition of $\tilde{\theta}^{(m)}$ and the second inequality due to the choice of $\delta$ (cf. 2nd step). Now, we apply the mean value theorem to achieve 
\begin{align*}
\tilde{\theta}^{(m)}(\hat{q}) - \tilde{\theta}^{(m)}({q}) 
&=  \max_{\xi \in [\tilde{F}_{q}(1) -  \frac{\varepsilon_1}{a} , \tilde{F}_{q}(1)]} \frac{1}{\xi} \cdot \frac{\varepsilon_1}{a}
+ a \left(-\log  \left(\tilde{F}_{q}(1) - \frac{\varepsilon_1}{a}\right)\right)\\
&= \frac{1}{\tilde{F}_{q}(1) - \frac{\varepsilon_1}{a}} \cdot \frac{\varepsilon_1}{a}
+ a \left(-\log  \left(\tilde{F}_{q}(1) - \frac{\varepsilon_1}{a}\right)\right).
\end{align*}
Using $e^{-m} \leq \tilde{F}_{q}(1)$ and (\ref{epshalf}), we arrive at
\begin{align*}
\tilde{\theta}^{(m)}(\hat{q}) - \tilde{\theta}^{(m)}({q}) 
&= \frac{1}{e^{-m} - \frac{\varepsilon_1}{a}} \cdot \frac{\varepsilon_1}{a}
+ a \left(-\log  \left(e^{-m} - \frac{\varepsilon_1}{a}\right)\right)  \leq \frac{\varepsilon}{2} + \frac{\varepsilon}{2}.
\end{align*}
This proves property B.

Thus, for each $m \geq 0$ we may uniquely extend $\tilde{\theta}^{(m)}$ to a continuous function $\theta^{(m)}: T^m \rightarrow [1,m]$. 
Naturally, this extension $\{\theta^{(m)}\}_{m \geq 0}$ inherits the properties  (\ref{fsymmetry}) and (\ref{fconsistency}) from $\{\tilde{\theta}^{(m)}\}_{m \geq 0}$.  Thus, it defines a function $\theta$ on $\finite(T)$ via (\ref{fcorrespondency}). By construction, this extension respects the three properties of theorem \ref{ecf_CA}, i.e. $\theta$ defines indeed the extremal coefficient function of a max-stable random field $X^*$ on $T$ as constructed in theorem \ref{ecf_CA}.  As mentioned above, $f=\chi$ is the corresponding extremal correlation function $\chi$ of this random field. Because of theorem \ref{cty_chitheta}, the random field $X^*$ is stochastically continuous. 
\end{proof}

\section{Inclusions of convex sets of functions}\label{inclusions}

\subsection{Extremal coefficient functions}

We consider the following sets of functions on $\finite(T)$ in this section:
\begin{align*}
  \maxecf(T) &:= \left\{ M \mapsto \theta(M)=
  \frac{ \log \PP(\max_{t \in M} X_t \leq x)}{ \log \PP(X_t \leq x)}
  \pmid \begin{array}{l} \text{$X$ a {\em max-stable} r.f. on $T$} \\ \text{with std. Fr{\'e}chet marginals} \end{array} \right\}\\ \\ 
  \limecf(T) &:= \left\{ M \mapsto \lim_{x \uparrow \infty} 
  \frac{  \log \PP(\max_{t \in M} Z_t \leq x)}{ \log \PP(Z_t \leq x)}
  \pmid \begin{array}{l} \text{$Z$ a r.f. on $T$ with the same} \\ \text{one-dimensional marginals}  \\ \text{and upper endpoint $\infty$,} \\ \text{{\em such that the limits exist}} \end{array} \right\}\\ \\ 
  \binecf(T) &:= \left\{ M \mapsto 
  \frac{  \PP(\max_{t \in M} Y_t = 1)}{ \PP(Y_t = 1)}
  \pmid \begin{array}{l} \text{$Y$ a r.f. on $T$ with the same} \\ \text{one-dimensional marginals}  \\ \text{{\em taking values in $\{0,1\}$,}} \\ \text{such that $\EE Y_t\neq 0$} \end{array} \right\}
  \end{align*}

In particular, we denote $\maxecf(T)$ the set of extremal coefficient functions of max-stable random fields on $T$, which is already characterized by theorem \ref{ecf_CA} as the set of completely alternating functions $\theta$ on $\finite(T)$ with $\theta(\emptyset)=0$ and $\theta(\{t\})=1$ for $t \in T$. 

\begin{proposition}\label{ecf_inclusions}
The following inclusions hold:
\begin{enumerate}[a)]
\item $\cl(\binecf(T)) = \maxecf(T) = \limecf(T)$
\item $\binecf(T) = \{ \, \theta \in \maxecf(T) \pmid \theta \text{ bounded } \}$
\end{enumerate}
where the closure is meant with respect to pointwise convergence.
\end{proposition}
\begin{proof}
\begin{enumerate}[a)]
\item \label{ecf_inclusions_proof_of_a}
Firstly, let $F \in \binecf(T)$. Then  $F$ is a normalized capacity functional (cf. (\ref{cap})), such that $F(\emptyset)=0$ and $F(\{t\})=1$. Therefore, $F$ is completely  alternating (cf. theorem \ref{cap_CA}) and satisfies all conditions of theorem \ref{ecf_CA}, i.e. we have $F \in \maxecf(T)$. This shows
\begin{align}\label{bininmax}
\binecf(T) \subset \maxecf(T)
\end{align}
Secondly, let $F \in \limecf(T)$ for some random field $Z$. Similarly to (\ref{theta_lHopital}) it is an easy exercise to show that
  \begin{align*}
    F(M) &= \lim_{x \uparrow  \infty} \frac{1-\PP\left(\max_{t \in M} Z_t \leq x \right)}{1-\PP\left(Z_t \leq x \right)}
    = \lim_{x \uparrow  \infty}
    \frac{\PP\left(\max_{t \in M}\Eins_{Z_t \geq x} = 1 \right)}{\PP\left(\Eins_{Z_t \geq x} = 1\right)}
     \end{align*} 
Therefore, $F \in \cl(\binecf(T))$. This shows
\begin{align}\label{liminclbin}
\limecf(T) \subset \cl(\binecf(T))
\end{align}
Additionally, we have by definition that {$\maxecf(T) \subset \limecf(T)$} and $\maxecf(T)$ is closed (cf. corollary \ref{ecf_closed}). So, all inclusions follow from
\begin{align*}
 \maxecf(T) \subset \limecf(T) \stackrel{(\ref{liminclbin})}{\subset} \cl(\binecf(T)) \stackrel{(\ref{bininmax})}{\subset} \cl(\maxecf(T)) = \maxecf(T).
\end{align*}
\item Let $F \in \binecf(T)$ and $Y$ a corresponding random field. Then $F$ is bounded by $\frac{1}{\EE Y_t}$. Let $\theta \in \maxecf(T)$ be bounded, say by $K$. Clearly, $K \geq \theta(\{t\}) = 1$.  Set $C(A):=\theta(A)/K$. Then $C$ satisfies all requirements of theorem \ref{cap_CA} and defines a  binary field $Y$ as in the definition of $\binecf(T)$, such that $\PP(\max_{t \in M} Y_t = 1) / \PP(Y_t=1)=\theta(M)$. \qedhere 
\end{enumerate}
\end{proof}

In fact, all involved sets are convex:

\begin{proposition}\label{ecf_convexity}
The sets $\binecf(T)$ and $\maxecf(T)=\limecf(T)$ are convex.
\end{proposition}
\begin{proof}
\begin{enumerate}[a)]
\item $\binecf(T)$ is convex: Let $Y_F$ and $Y_G$ be binary random fields on $T$corresponding to $F,G\in \binecf(T)$ and  $p_F:= \EE (Y_F)_t > 0$, $p_G:= \EE (Y_G)_t > 0$. Choose $U \sim \Bin(1,\alpha)$, independent of $Y_F$ and $Y_G$.
  Then $U Y_F  + (1- U) Y_G$ is binary and its capacity $C_\alpha: \finite(T) \rightarrow \RR$ is given in terms of the capacities $C_F$ and $C_G$ by
  \begin{align*}
    C_\alpha = \frac{\alpha p_F}{\alpha p_F + (1-\alpha)p_G} \, C_F + \frac{(1-\alpha)p_G}{\alpha p_F + (1-\alpha)p_G} \, C_G.
  \end{align*}
  Now, for every $\beta \in (0,1)$ there is a unique $\alpha\in(0,1)$ such that $\frac{\alpha p_f}{\alpha p_f + (1-\alpha)p_g}=\beta$.
\item $\maxecf(T)$ is convex: 
  Consider the the max-combination $X = \max(\alpha X_1, (1-\alpha) X_2)$
  of two independent random fields $X_1$ and $X_2$ on $T$ with standard Fr{\'e}chet marginals and extremal coefficient functions $\theta_1$ and $\theta_2$, respectively. Then the extremal coefficient function $\theta$ of $X$ is $\alpha \theta_1 + (1-\alpha)\theta_2$. \qedhere
 \end{enumerate} 
\end{proof}

\begin{remark}
Alternatively, the convexity $\maxecf(T)=\limecf(T)$ follows directly from theorem \ref{ecf_CA} or we could argue that it must be convex as the closure of the convex set $\binecf(T)$. But the argument given here is a bit more constructive.
\end{remark}

\subsection{Extremal correlation functions}

Similarly to the previous section, we define  the following sets of functions on $T \times T$:
\begin{align*}
  \maxcorr(T) &:= \left\{ (s,t) \mapsto \chi(s,t)= \lim_{x \uparrow \infty} \PP(X_s \geq x \mid X_t \geq x) 
  \pmid \begin{array}{l} \text{$X$ a {\em max-stable} r.f. on $T$} \\ \text{with std. Fr{\'e}chet marginals} \end{array} \right\}\\ \\ 
  \limcorr(T) &:= \left\{ (s,t) \mapsto \lim_{x \uparrow \infty} \PP(Z_s \geq x \mid Z_t \geq x) 
  \pmid \begin{array}{l} \text{$Z$ a r.f. on $T$ with same} \\ \text{one-dimensional marginals}  \\ \text{with upper endpoint $\infty$,} \\ \text{{\em such that the limits exist}} \end{array} \right\}\\ \\ 
  \bincorr(T) &:= \left\{ (s,t) \mapsto \PP(Y_s = 1 \mid Y_t = 1 ) 
  \pmid \begin{array}{l} \text{$Y$ a r.f. on $T$ with same} \\ \text{one-dimensional marginals}  \\ \text{{\em taking values in $\{0,1\}$,}} \\ \text{such that $\EE Y_t\neq 0$} \end{array} \right\}
  \end{align*}

All of these are sets consist of symmetric, positive definite functions on $T \times T$ with values in $[0,1]$, which take the value $1$ on the diagonal. 

\begin{proposition}\label{corr_inclusions}
The following inclusions hold:
\begin{enumerate}[a)]
\item $\bincorr(T) \subset \maxcorr(T) \subset  \limcorr(T)$
\item $\limcorr(T)=\cl(\bincorr(T))=\cl(\maxcorr(T))$
\end{enumerate}
where the closure is meant with respect to pointwise convergence.\\
If $T$ is a separable metric space, we have additionally
\begin{enumerate}[c)]
\item $\{ \chi \in \maxcorr(T) : \chi \text{ continuous} \} = \{ \chi \in \limcorr(T) : \chi \text{ continuous} \}$.
\end{enumerate}
\end{proposition}

\begin{remark}
In particular if $T$ is countable, we have $\maxcorr(T)=\limcorr(T)$ (cf. corollary \ref{countable_corr_closed}). It would be interesting to know whether also in the general situation $\maxcorr(T)=\limcorr(T)$ is correct. 
\end{remark}

\begin{proof}[Proof of proposition \ref{corr_inclusions}]
  \begin{enumerate}[a)]
  \item \label{bin_in_coef}
     Let $(Y_t)_{t \in T}$ be a binary random field with  same one-dimensional marginals and $\EE Y_t \neq 0$. Then $\PP(Y_s = 1 \mid Y_t = 1 ) = 2-C(\{s,t\})/C(\{t\})$ in terms of the capacity functional $C$ of $Y$, just as $\chi(s,t)=2-\theta(\{s,t\})$ for the quantities that we defined for max-stable random fields (cf. (\ref{sumistwo})). Therefore, the assignment $C \mapsto \theta$ from capacity functionals of binary random fields to extremal correlation functions of max-stable random fields, which is given by $\theta(M):=C(M)/C(\{t\})$ proves $\bincorr(T) \subset \maxcorr(T)$ (cf. the proof of $\binecf(T)\subset \maxecf(T)$ in proposition \ref{ecf_inclusions}). The inclusion $\maxcorr(T) \subset  \limcorr(T)$ follows directly from their definitions.

\item Because of \ref{bin_in_coef}) it suffices to show $\limcorr(T)=\cl(\bincorr(T))$.
  By definition of $\limcorr(T)$ and $\bincorr(T)$ and considering the random fields $Y_t = \Eins_{Z_t \geq x}$ indexed by $x>0$, we have $\limcorr(T) \subset \cl(\bincorr(T))$. To show the reverse inclusion let $f_n \to f$ be a pointwise converging sequence with $f_n \in \bincorr(T)$ and $Y_n$ the corresponding binary random fields, which are chosen to be independent. For further computations we abbreviate $a_n(s,t):=\PP((Y_n)_s = 1, (Y_n)_t=1)$ and $b_n:=\PP((Y_n)_t=1) = \EE (Y_n)_t$ (which is independent from $t$). Note that by choice of $Y_n$ we have $0 \leq a_n(s,t) \leq b_n \leq 1$ and $0<b_n$ and that $\frac{a_n(s,t)}{b_n} \rightarrow f(s,t)$ as $n \to \infty$. We may assume that one of the following two cases applies:
\begin{description}
\item{\bf{First case:}} There exists a lower bound $b > 0$ such that $b_n>b$ for all $n \in \NN$.
\item{\bf{Second case:}} The sequence $b_n$ converves monotonously to $0$ as $n \to \infty$.
\end{description}
(Otherwise there exists a subsequence $(b_{n_k}) \rightarrow 0$ that converges monotonously to $0$ as $k \to \infty$ and we work with the sequence $(f_{n_k})_k$ instead of $(f_n)_n$ which still converges to $f$ pointwise.)

Set $X = N \cdot Y_N$, where $N$ is an independent random variable taking values in $\NN$ with $p_n:=\PP(N =n) = 2^{-2^n}$ for $n\ge 2$ and $p_1:=\PP(N=1) = 1 - \sum_{i=2}^\infty  2^{-2^i}$. Then the sequence $(p_n)_{n\in\NN}$ satisfies the property that $\sum_{n=m+1}^{\infty}\frac{p_n}{p_m} \leq 2^{-2^m} \rightarrow 0$ as $m \to \infty$. Now,  
\begin{align*}
\PP(X_s \geq x \mid X_t \geq x) 
=\frac{\sum_{n=\lceil x \rceil}^{\infty} a_n(s,t) \cdot p_n}{\sum_{n=\lceil x \rceil}^{\infty} b_n \cdot p_n}
=\frac{\frac{a_m(s,t)}{b_m}+\sum_{n=m+1}^{\infty} \frac{a_n(s,t)}{b_m} \cdot \frac{p_n}{p_m}}{1 +\sum_{n=m+1}^{\infty} \frac{b_n}{b_m} \cdot \frac{p_n}{p_m}}
\end{align*}
for $m=\lceil x \rceil$. In order to prove $\lim_{x \uparrow \infty} \PP(X_s \geq x \mid X_t \geq x)  = f(s,t)$, it suffices to show that $\lim_{m\to \infty} \sum_{n=m+1}^{\infty} \frac{b_n}{b_m} \cdot \frac{p_n}{p_m} = 0$, since $0\leq a_n(s,t)\leq b_n$. Therefore, let us consider the two cases: 
\begin{description}
\item{\bf{First case:}} Then $\sum_{n=m+1}^{\infty} \frac{b_n}{b_m} \cdot \frac{p_n}{p_m} \leq \frac{1}{b} \sum_{n=m+1}^{\infty} b_n \cdot \frac{p_n}{p_m}\leq \frac{1}{b} \sum_{n=m+1}^{\infty}  \cdot \frac{p_n}{p_m} \rightarrow  0$ by choice of $(p_n)_{n \in \NN}$.
\item{\bf{Second case:}} Then $\sum_{n=m+1}^{\infty} \frac{b_n}{b_m} \cdot \frac{p_n}{p_m} \leq \sum_{n=m+1}^{\infty} 1 \cdot \frac{p_n}{p_m} \rightarrow  0$ by choice of $(p_n)_{n \in \NN}$.
\end{description}
\item This is a direct consequence of theorem \ref{corr_cvgce_to_cts}.  \qedhere
\end{enumerate}
\end{proof}

Again, all involved sets are convex:

\begin{proposition}
All sets  $\maxcorr(T), \bincorr(T), \limcorr(T)$ are convex.
\end{proposition}
\begin{proof}
The convexity of $\bincorr(T)$ and $\maxcorr(T)$ follow from proposition \ref{ecf_convexity}, since we have $\PP(Y_s = 1 \mid Y_t = 1 ) = 2-\PP(\max(Y_s,Y_t)=1)/\PP(Y_t=1)$ for the binary random field $Y$ and $\chi(s,t)=2-\theta(\{s,t\})$ (cf. (\ref{sumistwo})).
The convexity of $\limcorr(T)=cl(\bincorr(T))=\cl(\maxcorr(T))$ follows from the convexity of $\bincorr(T)$ or $\maxcorr(T)$, since the closure does not affect the property of being convex.
\end{proof}

\paragraph{Further operations on the sets  $\maxcorr(T), \bincorr(T), \limcorr(T)$}

\begin{corollary}
Let $f \in \limcorr(T)$. Then $1-f$ satisfies for any Bernstein function $g$ the triangle inequality $g(1-f(s,t)) \leq g(1-f(s,r))+g(1-f(r,t))$.
\end{corollary}
\begin{proof}
This is a consequence of proposition \ref{corr_inclusions} and corollary \ref{eta_inequality}.
\end{proof}

\begin{corollary}\label{products}
\begin{enumerate}[a)]
\item Let $f_1$ and $f_2$ be functions both either from $\bincorr(T)$ or $\limcorr(T)$. Then the product $f_1 \cdot f_2$ belongs again to this set of functions.
\item Let $T$ be separable metric. Let $f_1$ and $f_2$ be functions from  
\begin{align*}
\{ \chi \in \maxcorr(T) \pmid \chi \text{ continuous}\} = \{ \chi(X) \in \maxcorr(T) \pmid X \text{ stochastically continuous}\}.
\end{align*}
Then the product $f_1 \cdot f_2$ belongs again to this set of functions.
\end{enumerate}
\end{corollary}
\begin{proof}
\begin{enumerate}[a)]
\item The statement is easy to prove for $\bincorr(T)$ by taking the product of random fields of independent random fields. For $\limcorr(T)$ it then follows from proposition \ref{corr_inclusions}.
\item Since $\maxcorr(T) \subset \cl(\bincorr(T))$ (cf. proposition \ref{corr_inclusions}) we may choose sequences $f_{n,1},f_{n,2} \in \bincorr(T)$ which converge pointwise to $f_1$ and $f_2$, respectively. Now, for each $n$ the product $ f_{n,1}\cdot f_{n,2}$ is again an element of $\bincorr(T) \subset \maxcorr(T)$ and pointwise $f_{n,1}\cdot f_{n,2}$ converges to $f_1 \cdot f_2$. Now, the claim follows from theorem \ref{corr_cvgce_to_cts}. \qedhere
\end{enumerate}
\end{proof}

\begin{remark}
In particular the set $\limcorr(T)$ is closed under convex combinations, products and pointwise limits, whereas the set
\begin{align*}
\{ \chi \in \maxcorr(T) \pmid \chi \text{ continuous}\} = \{ \chi(X) \in \maxcorr(T) \pmid X \text{ stochastically continuous}\}
\end{align*}
is closed under convex combinations, products and pointwise limits as long as the limit function is continuous (cf. theorem \ref{corr_cvgce_to_cts}).
This allows for further operations built from these. We may apply convergent power series with positive coefficients: For instance, if $\chi \in \maxcorr(T)$ is continuous, then also  $1-(1-\chi)^{q} \in \maxcorr(T)$ for $q \in (0,1)$. 
Note that corollary \ref{theta_ops} is in accordance with these operations, since corollary \ref{theta_ops} implies that for any $\chi \in \maxcorr(T)$ and any Bernstein function $g$ with $g(0)=0$ and $g(1)=1$, the function $2-g(2-\chi)$ is again in $\maxcorr(T)$. However, the above is slightly more general for continuous members of $\maxcorr(T)$. For instance, we get that $1-g(1-\chi) \in \maxcorr(T)$, if $\chi \in \maxcorr(T)$ is continuous.
Finally we want to remark that, although these operations (convex combinations, products, pointwise limits) are well-known operations on the set of positive definite functions, it has not been obvious so far that they are compatible with the restriction to (continuous) extremal correlation functions of max-stable random fields. 
\end{remark}

\section{The stationary case}\label{stationary}

Most of the results stated before, are compatible with group actions on the location space $T$ and we could formulate a version which incorporates this group action. For example if we let $T=\RR^d$ and focus on {\em stationary} random fields on $\RR^d$, theorem \ref{ecf_CA} has  the following version:

\begin{theorem} 
  \begin{enumerate}[a)]
  \item
    The function $\theta: \finite(T) \rightarrow \RR$ is the extremal coefficient function of a stationary max-stable random field on $\RR^d$ with standard Fr{\'e}chet marginals if and only if the following four conditions are satisfied:
    \begin{enumerate}[(i)]
    \item  $\theta(\emptyset) = 0$
    \item  $\theta(\{t\}) = 1$ for all $t \in \RR^d$
    \item  $\theta$ is completely alternating.
    \item  $\theta(M+t)=\theta(M)$ for any $M \in \finite(\RR^d)\setminus\{\emptyset\}$ and any $t \in \RR^d$.
    \end{enumerate}
  \item  If these conditions are satisfied, the following choice of coefficients 
    \begin{align*}
      \tau^M_L :=  - \Delta_{\{t_1\}} \dots \Delta_{\{t_l\}} \theta &(M \setminus L) =  
      \sum_{I \subset L} (-1)^{|I|+1} \theta((M \setminus L) \cup I)\\
      & \quad \forall \, M \in \finite(T)\setminus\{\emptyset\},\, \emptyset \neq  L = \{t_1,\dots,t_l\} \subset M
    \end{align*}
    for model (\ref{model}) defines a stationary max-stable random field $X^*$ on $\RR^d$ with standard Fr{\'e}chet marginals which realizes $\theta$ as its own extremal coefficient function $\theta^*$.
  \end{enumerate}
\end{theorem}

Similarly, the results for the bivariate quantities transfer, in particular theorem \ref{corr_cvgce_to_cts} and corollary \ref{products}, which leads to the respective operations on the set of extremal correlation functions of stationary max-stable random fields. 

Further, the triangle inequality
\begin{align*}
g \circ \eta(s \pm t) \leq g \circ \eta(s) + g \circ \eta(t) \qquad \text{for $s,t \in \RR^d$}
\end{align*}
of the function $\eta=1-\chi$ of a stationary max-stable random field on $\RR^d$ and any Bernstein function $g$ (cf. corollary \ref{eta_inequality}) can be found already in the literature for some specific functions $g$:
The three choices $g(x)=\log(1+x)$, $g(x)=(1+x)^\tau-1$ for $0 < \tau < 1$ and $g(x)=1-(1+x)^\tau$ for $\tau <0$ as Bernstein functions yield precisely the set of inequalities derived in \cite{cooleyetalii_06} Proposition 4. Therein a madogram approach is used instead.

This triangular inequality where $g=\id$ is also the reason, why such functions $\eta$ (and equivalently $\chi=1-\eta$) may not be differentiable at the origin unless they are constant \cite{schlathertawn_03}. This is also true if we considered instead of $\RR^d$ any other abelian connected Lie group acting by translation on itself. For more implications of the triangular inequality see \cite{markov_95}.

Finally, we want to mention that if we consider only stationary random fields on $\RR^d$ some of the inclusions in Proposition \ref{ecf_inclusions} and \ref{corr_inclusions}) will be proper:
\begin{proposition}
The following inclusions are proper inclusions:
\begin{align}
\label{chi_stat}
\left\{ \chi(Y) \in \bincorr(\RR^d)  \pmid \text{\normalfont $Y$ stationary}\right\} 
&\subsetneq \left\{ \chi(X) \in \maxcorr(\RR^d)  \pmid \text{\normalfont $X$ stationary}\right\}\\
\notag
\left\{ \theta(Y) \in \bincorr(\RR^d)  \pmid \text{\normalfont $Y$ stationary}\right\} 
&\subsetneq \left\{ \theta(X) \in \maxcorr(\RR^d)  \pmid \text{\normalfont $X$ stationary}\right\}
\end{align}
\end{proposition}
\begin{proof}
It suffices to show that the inclusion (\ref{chi_stat}) is proper. We will show that the l.h.s. of (\ref{chi_stat}) may not contain functions with compact support, whereas the r.h.s. does contain such functions: 

Suppose the l.h.s. contains a function $f:\RR^d\rightarrow [0,1]$, $f(t)=\PP(Y_t=1 \mid Y_0=1)$ with compact support for some stationary random field $Y$ on $\RR^d$ with values in $\{0,1\}$ and $p:=\EE Y_0 > 0$. Then the normalized covariance function $C(h):=\Cov(Y_t,Y_0)/p=f(t)-p$ has to be positive definite on $\RR^d$. Further, set $\gamma(t):=C(0)-C(t)=1-f(t)$, which is constantly $1$ away from some compact set around the origin because $f$ has compact support. Due to \cite{gneitingsasvarischlather_01}, theorem 6, we conclude that $1>1-p=C(0)\geq \lim_{s \to \infty}   (2s)^{-d} \int_{[-s,s]^d} \gamma(t) \D t = 1$, which leads to a contradiction.

On the other hand, the r.h.s. of (\ref{chi_stat}) contains functions of compact support. For instance, let $A\subset \RR^d$ be a compact subset of $\RR^d$ and let $\mu=\int \Eins_A(s) \D s$. Let $\Pi$ be a Poisson point process on $(0,\infty) \times \RR^d$ with intensity $\mu^{-1} u^{-2} \D u \D s$. Then the random field  $X_t = \sup_{(u,s) \in \Pi} u \Eins_A(t-s) $ on $\RR^d$ is stationary and max-stable \cite{schlather_02} and has extremal correlation function $\chi(t)=\mu^{-1} \int \Eins_A(s-t)\Eins_A(s) \D s$, which has compact support. 
\end{proof}

\noindent{\bf Acknowledgment}
We would like to thank Ilya Molchanov for inspiring discussions. Financial support for the first author by the German Research Foundation DFG through the Research Training Group 1023 is gratefully acknowledged.

\bibliographystyle{plainnat}
\bibliography{literature}

\end{document}